\documentclass[a4paper,12pt]{amsart}
\usepackage[english]{babel}
\usepackage{amsmath,amsthm,amssymb,amsfonts, mathabx}
\usepackage{multicol}
\usepackage[titletoc]{appendix}
\usepackage{soul}

\usepackage[pdftex]{color}
\usepackage[bookmarks=true,hyperindex,pdftex,colorlinks, citecolor=blue,linkcolor=blue, urlcolor=blue]{hyperref}
\usepackage[dvipsnames]{xcolor}
\usepackage{mathtools}
\usepackage{graphicx}
\usepackage{enumitem}
\usepackage{pgf}
\usepackage{tikz-cd}
\usepackage[normalem]{ulem}

\usepackage{tikz}
\usetikzlibrary{mindmap,backgrounds, calc}
\usetikzlibrary{matrix,chains,scopes,positioning,arrows,fit}
\usetikzlibrary{positioning, shapes, patterns}
\usetikzlibrary{automata} 
\usetikzlibrary{shapes.geometric, arrows, arrows.meta}
\usetikzlibrary{positioning,decorations.markings}

\usetikzlibrary{matrix}\parskip=1ex
\newtheorem{theorem}{Theorem}[section]
\newtheorem{lemma}[theorem]{Lemma}

\newtheorem{proposition}[theorem]{Proposition}
\newtheorem{cor}[theorem]{Corollary}

\theoremstyle{definition}

\newtheorem{example}[theorem]{Example}

\newtheorem{remark}[theorem]{Remark}
\numberwithin{equation}{section}

\newcommand{\C}{\mathbb{C}}
\newcommand{\N}{\mathbb{N}}

\newcommand{\Lin}{\mathcal{L}}
\newcommand{\Ho}{\mathcal{H}}
\newcommand{\Hb}{\mathcal{H}_b}
\newcommand{\M}{\mathfrak{M}}
\newcommand{\A}{\mathcal{A}}
\newcommand{\Pol}{\mathcal{P}}

\newcommand{\nn}[1]{{\left\vert\kern-0.25ex\left\vert\kern-0.25ex\left\vert #1 
		\right\vert\kern-0.25ex\right\vert\kern-0.25ex\right\vert}}

\renewcommand{\geq}{\geqslant}
\renewcommand{\leq}{\leqslant}

\newcommand{\NA}{\operatorname{NA}}

\newcommand{\pten}{\ensuremath{\widehat{\otimes}_\pi}}

\newcommand{\eps}{\varepsilon}
\newcommand{\Nsten}{\widehat{\otimes}_{N, s, \pi}\,}

\begin{document}


\title[Daugavet property for Banach algebras...]{Daugavet property of Banach algebras of holomorphic functions and norm-attaining holomorphic functions}

\author[Jung]{Mingu Jung}
\address[Jung]{School of Mathematics, Korea Institute for Advanced Study, 02455 Seoul, Republic of Korea \newline
\href{http://orcid.org/0000-0003-2240-2855}{ORCID: \texttt{0000-0003-2240-2855} }}
\email{\texttt{jmingoo@kias.re.kr}}
\urladdr{\url{https://clemg.blog/}}

\thanks{The author was supported by NRF (NRF-2019R1A2C1003857), by POSTECH Basic Science Research Institute Grant (NRF-2021R1A6A1A10042944) and by a KIAS Individual Grant [MG086601] at Korea Institute for Advanced Study.}


\date{\today}


\begin{abstract}
We show that the duals of Banach algebras of scalar-valued bounded holomorphic functions on the open unit ball $B_E$  of a Banach space $E$ lack weak$^*$-strongly exposed points. Consequently, we obtain that some Banach algebras of holomorphic functions on an arbitrary Banach space have the Daugavet property which extends the observation of P.~Wojtaszczyk \cite{Wo}. Moreover, we present a new denseness result by proving that the set of norm-attaining vector-valued holomorphic functions on the open unit ball of a dual Banach space is dense provided that its predual space has the metric $\pi$-property. Besides, we obtain several equivalent statements for the Banach space of vector-valued homogeneous polynomials to be reflexive, which improves the result of J.~Mujica \cite{Mujica1}, J.~A.~Jaramillo and L.~A.~Moraes \cite{JM}. As a byproduct, we generalize some results on polynomial reflexivity due to J.~Farmer \cite{Farmer94}. 
\end{abstract}
\maketitle


\section{Introduction}

Throughout this paper, the letters $E$ and $F$ represent complex Banach spaces. Let us denote by $B_E$ the \textbf{open} unit ball and by $\overline{B}_E$ the \textbf{closed} unit ball. Using the classical notation, let us denote by $\Pol(^N E, F)$ the Banach space of all $N$-homogeneous continuous polynomials from $E$ to $F$. In particular, $\Pol (^1 E, F)$ is the space of all bounded linear operators from $E$ to $F$, which we will denote by $\Lin (E, F)$. It is well-known that the space $\Ho^\infty (B_E; F)$ of bounded holomorphic functions on $B_E$ into $F$ becomes a Banach space when it is endowed with the norm $\| f \| = \sup_{x \in B_E} \| f (x) \|$. 
When $F = \mathbb{C}$, then we write simply $\Pol (^N E)$ and $\Ho^\infty (B_E)$ instead of $\Pol(^N E; \mathbb{C})$ and $\Ho^\infty (B_E; \mathbb{C})$, respectively. It is well-known that $\Pol(^N E)$ and $\Ho^\infty (B_E)$ are dual Banach spaces. Indeed, recall the following result due to R.~Ryan \cite{Ryan_thesis}: For each Banach space $F$ and each $P$ in $\Pol(^N E; F)$, there is an unique operator $T_P \in \Lin (\Nsten E, F)$ such that $T_P \circ q_N = P$, where $q_N \in \Pol(^N E; \, \Nsten E)$ is given by $q_N (x) = \otimes^N x := x \otimes \cdots \otimes x$, and the correspondence $P \in \Pol(^N E, F) \longleftrightarrow T_P \in \Lin (\Nsten E, F)$ is an isometric isomorphism. Here, $\Nsten E$ denotes the $N$-fold (completed) symmetric tensor product of $E$. For bounded holomorphic mappings, J.~Mujica observed in \cite{Mujica} that, given a Banach space $E$, there are a (unique) Banach space $G^\infty (E)$ and a (norm one) mapping $g_E \in \Ho^\infty (B_E; G^\infty(E))$ satisfying the following universal property: For each Banach space $F$ and each $f \in \Ho^\infty (B_E; F)$, there exists a unique $T_f \in \Lin (G^\infty (E), F)$ such that $T_f \circ g_{E} = f$, and the correspondence $f \in \Ho^\infty (B_E; F) \longleftrightarrow T_f \in \Lin (G^\infty (E), F)$ is an isometric isomorphism. Explicitly, $G^\infty (E)$ is given as 
\[
\left\{ \varphi \in \Ho^\infty (B_E)^* \,: \, \varphi \vert_{(B_{\Ho^\infty (B_E)}, \tau_c)} \text{ is continuous} \right\},
\]
where $\tau_c$ is the compact-open topology.

One of the main motivations for the present work stems from the paper \cite[Theorem 3.1]{AAGM}, where the authors, arguing as in \cite[Theorem 2.7]{CK96}, obtained that if a Banach space $E$ the Radon-Nikod\'ym property (for short, RNP), then for every $N \in \N$, every $\eps >0$ and every $f \in \A_u (B_E; F)$, there exists $P \in \Pol(^N E; F)$ such that $\|P\| < \eps$ and $f + P$ attains its norm for every Banach space $F$. Here, $\A_u (B_E; F)$ denotes the closed subspace of $\Ho^\infty (B_E; F)$ consisting of elements which are uniformly continuous on $B_E$. Recall that an element $f$ in $\A_u (B_E; F)$ (or in $\Pol(^N E; F)$) is said to \emph{attain its norm} if there exists $x_0 \in \overline{B}_E$ such that $\| f (x_0) \| = \| f \|$. One of direct consequences of Mujica's linearization result shows that if $f \in \A_u ( B_E; F) \subset \Ho^\infty (B_E; F)$ attains its norm at some $x_0$, then its corresponding bounded linear operator $T_f \in \Lin (G^\infty (E), F)$ also attains its norm at $g_E (x_0)$. 
Thus, it can be seen that the RNP of a Banach space $E$ is somehow related to the denseness of norm-attaining elements in $\Lin (G^\infty (E), F)$ (for any Banach space $F$). Motivated by this relation and the work of J.~Bourgain \cite{Bourgain} on the Bishop-Phelps property, one may ask if the Banach space $G^\infty (E)$ has the RNP under some assumptions on $E$.

We proceed to describe the content of the paper. In Section \ref{denseness}, we prove  that dual spaces of $\Ho^\infty (B_E)$ and $\A_u (B_E)$ do not have weak$^*$-strongly exposed points (see Theorem \ref{theorem:predual_no_strongly_exposed_points}). This result, in particular, shows that the Banach space $G^\infty (E)$ fails to have the RNP for every Banach space $E$ (see Corollary \ref{cor:fails_RNP}). Furthermore, we obtain that $\Ho^\infty (V)$ and $\A_u (V)$ have the Daugavet property for any bounded open balanced convex subset $V$ of a Banach space $E$ (see Corollary \ref{cor:DP_holo}). This can be viewed as a generalization of the fact that $\Ho^\infty (V)$ or $\A_u  (V)$ has the Daugavet property whenever $V$ is an open and connected subset of $\mathbb{C}^N$, which was observed by P.~Wojtaszczyk \cite{Wo}. In Section \ref{sec:NA}, we give a different proof of the denseness of strongly norm-attaining $N$-homogeneous polynomials from $E$ into $F$ with the aid of  some tensor product theory provided that $E$ is a dual Banach space with the RNP and the approximation property. We also present a new denseness result which can be seen as weak$^*$-version of \cite[Theorem 3.3]{AAGM} by proving the denseness of norm-attaining weak$^*$ uniformly continuous holomorphic functions on a dual Banach space provided that the predual of the given Banach space has the metric $\pi$-property (see Theorem \ref{thm:metric_pi}). In Section \ref{Holub}, we focus on the reflexivity of the Banach space $\Pol(^N E; F)$ which is closely related to the existence of non-norm-attaining elements in $\Pol(^N E; F)$. We generalize in Theorem \ref{thm:holub_poly} the result \cite[Theorem 3.1]{Mujica1} of J.~Mujica on the reflexivity of $\Pol(^N E; F)$ by proving that for a reflexive Banach space $E$ with compact approximation property and an \emph{arbitrary} Banach space $F$ the following holds: $\Pol (^N E; F) = \Pol_w(^N E; F)$ $\Longleftrightarrow$ $\Pol (^N E; F) = \NA \Pol(^N E; F)$ $\Longleftrightarrow$ $(\Pol(^N E; F), \| \cdot \|)^* = (\Pol(^N E; F), \tau_c)^*$ $\Longleftarrow$ $\Pol (^N E; F)$ is reflexive, where $\NA \Pol(^N E; F)$ denotes the set of all norm-attaining $N$-homogeneous polynomials between $E$ and $F$ and $\tau_c$ is the compact-open topology. Furthermore, we also improve the result \cite[Theorem 1.5]{JM} of J.~A.~Jaramillo and L.~A.~Moraes on the reflexivity of the space of entire functions of bounded type (see Theorem \ref{thm:holub_holo}). As a byproduct of the results just mentioned, we observe in Section \ref{Farmer} that some results on polynomial reflexivity of $E$ in \cite{Farmer94} can be improved. For instance, we obtain an equivalent statement for every $N$-homogeneous scalar-valued polynomial on $E$ to be weakly sequentially continuous provided that $E$ is a reflexive Banach space with the compact approximation property (see Theorem \ref{thm:gen_Farmer_1}). 

\section{Daugavet property of Banach algebras of holomorphic functions}\label{denseness}






In this section, we will observe that the dual $\mathcal{A}^*$ of a certain Banach algebra $\A$ of bounded holomorphic functions do not have weak$^*$-strongly exposed points. Let us first explain some notation and definitions. First, recall that a unit vector $x$ in a Banach space $E$ is \emph{strongly exposed} if there is a unit vector $x^* \in E^*$ so that $x^* (x) = 1$ and given any sequence $(x_n) \subseteq B_E$ with $\text{Re}\ x^* (x_n) \rightarrow 1$ we can conclude that $x_n$ converges to $x$ in norm. If $E$ is a dual space and the element $x^*$ which strongly exposes at $x$ is in the predual of $E$, then we say that $x$ is \emph{weak$^*$-strongly exposed}. Next, we shall denote by $\A_\infty (B_E)$ the Banach algebra 
\[
\{ f \in \Ho^\infty (B_E) : f \text{ extends to the closed unit ball } \overline{B}_E \text{ continuously}\}. 
\]
It is known that $\A_u (B_E) \subset \A_\infty (B_E) \subset \Ho^\infty (B_E)$ while each inclusion is strict in general (see \cite[Section 12]{ACG91}). 
Notice that $\A_u (V), \A_\infty (V)$ and $\Ho^\infty (V)$ also can be defined in an obvious way for a bounded convex open subset $V$ of a Banach space $E$, and each of them is a Banach algebra when it is endowed with the supremum norm on $V$.

When $\mathcal{A}$ denotes either $\Ho^\infty (B_E)$, $\A_\infty (B_E)$ or $\A_u (B_E)$, we denote by $\M ( \mathcal{A})$ the \emph{spectrum} of the algebra $\mathcal{A}$ which consists of all homomorphisms, i.e., all continuous linear functionals on $\mathcal{A}$ that are also multiplicative. Let us notice that the spectrum $\M(\A)$ is always a compact set endowed with the Gelfand topology (that is, the restriction of the weak$^*$-topology of $\A^*$). 
We shall borrow some results from \cite{ADLM} on Gleason parts of the spectrum of the algebras of holomorphic functions. Given $\varphi$ and $\psi$ in $\M (\A)$, the distance $\| \varphi - \psi \| = \sup\{ |\varphi (f)-\psi(f)| : f \in \A, \, \| f \| \leq 1\}$ is called the \emph{Gleason metric}, and the following relation is well-known (see \cite[Theorem 2.8]{Bear}):
\[
\| \varphi -\psi \| = \frac{2-2 \sqrt{1-\rho(\varphi,\psi)^2}}{\rho(\varphi, \psi)},
\]
where $\rho$ is the \emph{pseudo-hyperbolic distance} i.e., $\rho (\varphi, \psi) := \sup \{ |\varphi(f)| : f \in \A, \, \| f \| \leq 1, \, \psi(f)=0\}$. For a more detailed account on the Gleason metric on spectra, we refer to \cite{ADLM, Bear} and the references therein. 

\begin{theorem}\label{theorem:predual_no_strongly_exposed_points}
Let $E$ be a Banach space, and let $\mathcal{A}$ denote $\A_u (B_E), \A_\infty (B_E)$ or $\Ho^\infty (B_E)$. Then $\A^*$ does not have weak$^*$-strongly exposed points.  
\end{theorem}

\begin{proof}
Assume that there is a weak$^*$-strongly exposed point $\varphi_0 \in \A^*$ with $\| \varphi_0 \| =1$ and let $f_0 \in \A$ be a norm one element which strongly exposes at $\varphi_0$. 
Notice that the element $f_0$ is not a constant function. 
Take a sequence $(x_n)$ in $B_E$ so that $|f_0 (x_n)| \rightarrow \| f_0 \| = 1$. 
By our assumption, there exists $\theta \in \mathbb{C}$ with $|\theta| = 1$ such that $\delta_{x_n} \rightarrow \theta \varphi_0$ in $\A^*$ in the norm topology. 
Notice that $\| x_n - x_m \| \leq \| \delta_{x_n} - \delta_{x_m}\|$ for any $n, m \in \N$ as the dual $E^*$ can be viewed as a subspace of $\A$. Thus, the sequence $(x_n)$ in $B_E$ is a Cauchy sequence. Let $x_0 \in \overline{B}_E$ be such that $x_n \rightarrow x_0$. If $x_0$ belongs to the interior $B_E$, then $\| f_0 \| = | f_0 (x_0) |$ and the Maximum Principle yields that $f_0$ must be constant (see \cite[Proposition 5.9]{Mujica}) which is a contradiction. Therefore, we get that $x_0 \in S_E$. 

On the other hand, since it is clear that $\delta_{x_n}$ is multiplicative for each $n\in \N$, we can actually view these elements as members of $\M (\A)$. If we put $\psi = \theta \varphi_0$, then $\psi$ is also an element of $\M (\A)$ as it is multiplicative as well (as a limit of point evaluation homomorphisms). Consider the canonical restriction mapping $\Gamma : \M (\A) \rightarrow \M (\A_u (B_E))$, i.e., $\Gamma(\varphi) = \varphi \vert_{\A_u (B_E)}$ for every $\varphi \in \M(\A)$. Note that $\Gamma (\delta_{x})$ is again an evaluation homomorphism on $\A_u (B_E)$; thus we keep the notation $\delta_x$ for each $x \in B_E$. Then we have that $\| \delta_{x_n} - \Gamma (\psi) \|_{\A_u (B_E)} \leq \| \delta_{x_n} - \psi \|_{\A} \rightarrow 0$ as $n \rightarrow \infty$. On the other hand, if $g \in \A_u (B_E)$ is given, then $\delta_{x_n} (g) = g(x_n) \rightarrow g(x_0) = \delta_{x_0} (g)$; which implies that $\delta_{x_n} \rightarrow \delta_{x_0}$ in $\M (\A_u (B_E))$ in the Gelfand topology. It follows that $\Gamma (\psi) = \delta_{x_0}$. However, \cite[Proposition 1.1]{ADLM} implies that $\| \delta_{x_n} - \delta_{x_0} \|_{\mathcal{A}_u (B_E)} = 2$ for all $n \in \N$, which is a contradiction. 
\end{proof}


Recall that a Banach space $E$ has the Radon-Nikod\'ym property (for short, RNP) if and only if every bounded closed convex subset of $E$ is the closed convex hull of its strongly exposed point \cite{Phelps}. Noting that the above proof of Theorem \ref{theorem:predual_no_strongly_exposed_points} shows that $G^\infty (E)$, the predual of the Banach space $\Ho^\infty (B_E)$, lacks strongly exposed points, we have the following corollary. 

\begin{cor}\label{cor:fails_RNP}
For any Banach space $E$, $G^\infty (E)$ does not have strongly exposed points; hence $G^\infty(E)$ fails to have the RNP. 
\end{cor} 

Recall that a Banach space $E$ is said to have the \emph{Daugavet property} if 
\begin{equation}\label{eq:DE}
\| T + I \| = 1 + \|T\|
\end{equation}
for every rank-one operator $T \in \Lin (E, E)$. For instance, it is known that, among others, $C(K)$ for a perfect compact Hausdorff space $K$, $L_1 (\mu)$, $L_\infty (\mu)$ for a non-atomic measure and the space of Lipschitz functions $\text{Lip}_0 (M)$ over a metrically convex space $M$ have the the Daugavet property. 
This property and related diametral diameter two properties have been extensively studied by many authors (see \cite{ALN13, ikw,kssw,rtv,werner} and the references therein). 

On the other hand, P.~Wojtaszczyk proved in \cite[Corollary 3]{Wo} that if a uniform algebra $A$ contains no nontrivial idempotents, then every weakly compact operator on $A$ to itself satisfies \eqref{eq:DE}. From this result, it is observed in the same paper that $\Ho^\infty (V)$ and $\A_u (V)$ have the Daugavet property whenever $V$ is an open and connected subset of $\mathbb{C}^N$ with $N \in \N$. Afterwards, D.~Werner observed in \cite[Theorem 3.3]{Werner} that if the Choquet boundary of a uniform algebra $A$ does not contain an isolated point, then every bounded linear operator $T$ on $A$ to itself which is weakly compact or factors through a subspace of $c_0$ satisfies \eqref{eq:DE} (hence $A$ has the Daugavet property). 
Let us also remark that it is recently observed by H.~J.~Lee and H.~J.~Tag that the function algebra $A(K, X)$ has the Daugavet property if and only if the Shilov boundary of its base algebra does not have isolated points (see \cite[Theorem 5.6]{LT}). However, to the best of our knowledge, it was unknown if $\A_u (B_E), \A_\infty (B_E)$ or $\Ho^\infty (B_E)$ has the Daugavet property for an arbitrary Banach space $E$. We present the following result which extends the aforementioned result of P.~Wojtaszczyk by showing that each of $\A_u (V), \A_\infty (V)$ and $\Ho^\infty (V)$ actually enjoys the Daugavet property whenever $V$ is a bounded convex balanced open subset of a Banach space $E$. 

Before we state and prove the result, we need some definitions. Given a Banach space $E$ and for $x^*, y^* \in \text{ext}(B_{E^*})$, define the equivalence relation $x^* \sim y^*$ if and only if $x^*$ and $y^*$ are linearly dependent elements of $\text{ext} (B_{E^*})$. We endow the quotient space $\text{ext} (B_{E^*}) / \sim$ with the quotient topology of the weak$^*$ topology. Recall also that a Banach space $E$ is said to be \emph{nicely embedded} into a $C_b (S)$ for a Hausdorff space $S$ if there exists an isometric embedding $J : E \rightarrow C_b (S)$ such that for every $s \in S$ the following are satisfied:
\begin{itemize}
\setlength\itemsep{0.3em}
\item[(N1)] for $p_s := J^* (\delta_s) \in X^*$, we have $\| p_s \| = 1$,
\item[(N2)] $\mathbb{K} p_s$ is an \emph{$L$-summand} in $X^*$, that is, there is a projection $\Pi$ from $E$ onto $\mathbb{K} p_s$ such that 
\[
\| \xi \| = \|\Pi(\xi)\| + \|\xi - \Pi(\xi)\| \text{ for every } \xi \in E. 
\]
\end{itemize}

\begin{cor}\label{cor:DP_holo}
Let $E$ be a Banach space and $V$ a bounded open balanced convex subset of $E$ containing $0$, and let $\mathcal{A}$ denote $\A_u (V), \A_\infty (V)$ or $\Ho^\infty (V)$. Then $\A$ has the Daugavet property.
\end{cor} 

\begin{proof} 
Let $\| \cdot \|_V$ be the Minkowski functional of the set $V$ which becomes a norm on $E$ with $B_{(E, \| \cdot \|_V)} = V$ \cite[Exercise 2.21]{FHHMZ}. Let $\A$ denote $\A_u (B_{(E, \| \cdot \|_V)}), \A_\infty (B_{(E, \| \cdot \|_V)})$ or $\Ho^\infty (B_{(E, \| \cdot \|_V)})$. 
Notice that $\A$ is isometrically isomorphic to its Gelfand transform image $\widehat{\A}$ which is a uniform algebra on the spectrum $\M (\A)$ (see, for instance, \cite[Proposition 2]{HanJuLee_solo}). 
For simplicity, put $X = \widehat{\A}$. Notice from \cite[Proposition 2.2]{BGM2006} that if $x^*$ is an extreme point of $B_{X^*}$ such that its equivalence class is an isolated point in the quotient space $\text{ext} (B_{X^*}) / \sim$, then $x^*$ turns to be a weak$^*$-strongly exposed point of $B_{X^*}$. 
Therefore, by Theorem \ref{theorem:predual_no_strongly_exposed_points}, we conclude that $\text{ext} (B_{X^*}) / \sim$ does not contain any isolated point. 
Next, as $X$ is a uniform algebra, we have that $X$ is nicely embedded in $C_b (\partial X)$ with respect to the isometric embedding $J : X \rightarrow C_b (\partial X)$ given by $J(f) := f \vert_{\partial X}$, where $\partial X$ denotes the Choquet boundary of $X$ (see, for instance, the proof of \cite[Theorem 3.3]{Werner}). Now, due to \cite[Proposition 3.7]{MRZ}, it suffices to prove the following: given $f \in S_X$ and $\eps>0$, the set $L = \{ s \in \partial X  : | p_s (f) | > 1 - \eps  \}$ is infinite, where $p_s = J^* (\delta_s) = \delta_s \vert_X \in X^*$. Assume to the contrary that the set $L$ is finite. Then, since $\text{ext} (B_{X^*}) = \mathbb{T} \{ \delta_s \vert_X : x \in \partial X \}$, this would imply that the set $\{ e^* \in \text{ext} (B_{X^*}) : |e^* (f)| > 1-\eps \}$ is a finite set up to rotation. Thus, only finitely many pairwise linearly independent points belong to $\text{ext} (B_{X^*}) \cap \{ x^* \in B_{X^*} : \text{Re}\,  x^* (x) > 1- \eps\}$. This implies that $\text{ext} (B_{X^*}) / \sim$ contains an isolated point, which is a contradiction. 
\end{proof} 


\section{On norm-attaining holomorphic functions}\label{sec:NA}

A famous result of J.~Bourgain \cite{Bourgain} shows that if a Banach space $E$ has the RNP, then $E$ has Lindenstrauss property A. Recall from \cite{Lin} that $E$ is said to have \emph{Lindenstrauss property A} if $\NA (E, F)$ is dense in $\Lin (E,F)$ for any Banach space $F$, where $\NA (E, F)$ is the set of all norm-attaining bounded linear operators from $E$ to $F$. As a matter of fact, Bourgain's result \cite[Theorem 7]{Bourgain} shows that $E$ has the RNP if and only if every equivalent renorming of $E$ has Lindenstrauss property A. However, Lindenstrauss property A does not imply the RNP in general. Indeed, W.~Schachermayer observed in \cite[Theorem 4.4]{schaalp} that every weakly compactly generated Banach space can be renormed to have the, so called, \emph{property $\alpha$} which implies Lindenstrauss property A. In particular, there is a renorming of $c_0$ which has the Lindenstrauss property A (but it fails to have the RNP). Nevertheless, as the first result in this section, we shall present a new example of such Banach spaces by using some results on norm-attaining polynomials.

Notice first the following simple observation: If $P \in \Pol(^N E; F)$ attains its norm at some $x_0$ in $\overline{B}_E$, i.e., $P \in \NA \Pol(^N E; F)$, then its corresponding bounded linear operator $T_P$ from $\Nsten E$ to $F$ attains its norm at $\otimes^N x_0$ due to the isometric isomorphic correspondence between $\Pol(^N E; F)$ and $\Lin (\Nsten E, F)$ (see \cite{Ryan_thesis}). However, the converse is not true in general. Indeed, it is observed in \cite[Corollary 4.4]{AAGM} that $\NA \Pol(^N d_{*} (w, 1); \mathbb{C})$ is not dense in $\Pol(^N d_{*} (w, 1); \mathbb{C})$ where $d_{*} (w, 1)$ is the predual of the Lorentz sequence space, with $w \in \ell_2 \setminus \ell_1$ and $N \geq 2$. In contrast, $\NA (\Nsten E, \mathbb{C})$ is always dense in $\Lin (\Nsten E, \mathbb{C})$ thanks to the Bishop-Phelps theorem \cite{BP}.



\begin{example}\label{theorem_counterexample_propertyA}
Let $E$ be the Banach space constructed by J.~Bourgain and G.~Pisier in \cite{Bourgain_Pisier} (see also \cite{DFS}). In the papers just mentioned, it is observed that $E$ has the RNP and that $E \pten E$ contains an isomorphic copy of $c_0$, where $E \pten E$ denotes the full $2$-fold tensor product of $E$. We claim that for any $N \geq 2$, $\widehat{\otimes}_{N,s,\pi} E$ has the Lindenstrauss property A but fails to have the RNP. As a matter of fact, according to the proof of \cite[Theorem 5]{DFS}, $E$ contains a copy of $\ell_2$ and 
\[
\sup_{1 \leq i \leq n} |a_i| \leq \left\| \sum_{i=1}^n a_i e_i\otimes e_i \right\|_{E \pten E} \leq C \sup_{1 \leq i \leq n} |a_i|, 
\]
where $(e_n)$ is the unit coordinate vector basis of $\ell_2$ inside of $E$, for some $C \geq 1$ (which does not depend on $n$). As $(e_n \otimes e_n)$ is already a subset of $\widehat{\otimes}_{2,s,\pi} E$, and as the symmetric projective norm is  equivalent to the full tensor projective norm on $\widehat{\otimes}_{2,s,\pi} E$ (see, for instance, \cite[Section 2]{Floret}), we can deduce that $\widehat{\otimes}_{2,s,\pi} E$ contains an isomorphic copy of $c_0$ as well. In particular, $\widehat{\otimes}_{2,s,\pi} E$ fails to have the RNP, which implies that $\Nsten E$ does not have the RNP for $N \geq 2$ as it contains a complemented subspace which is isomorphic to $\widehat{\otimes}_{2,s,\pi} E$ (see \cite[Corollary 4]{Blasco}). On the other hand, since $E$ has the RNP, we know from \cite[Theorem 3.1]{AAGM} that $\NA\Pol(^N E; F)$ is dense in $\Pol(^N E; F)$ for any $N \in \N$ and any Banach space $F$. Thus, $\widehat{\otimes}_{N,s,\pi} E$ has Lindenstrauss property A. Indeed, if $T = T_P \in \Lin (\widehat{\otimes}_{N,s,\pi} E, F)$ with $P \in \Pol(^N E ; F)$ and $\eps >0$ are given, then there exists $Q \in \NA \Pol(^N E; F)$, say $Q$ attains its norm at $x_0$, such that $\| P - Q \| < \eps$. Then $\| T - T_Q \| < \eps$ and $T_Q$ attains its norm at $\otimes^N x_0$.
\end{example}

Motivated by the aforementioned result of Bourgain, it is observed in \cite[Theorem 4.2]{CLS} that if $E$ is a Banach space with RNP, then given $N \in \N$, the set of strongly norm-attaining $N$-homogeneous polynomials is dense in $\Pol(^N E; F)$ for any Banach space $F$. 
Recall that a function $f \in \mathcal{A}_u (B_E; F)$ is said to be \emph{strongly norm-attaining} if there is $x_0 \in S_E$ such that whenever $\lim_n \| f(x_n)\| = \|f\|$ for a sequence $(x_n)$ in $\overline{B}_E$, it has a subsequence which converges to $\theta x_0$ for some $\theta \in \mathbb{C}$ with $|\theta|=1$. However, assuming the domain space to be a dual space with the approximation property (for short, AP) as well, we can conclude the same denseness result but in a different way by using some tensor product theory.

\begin{theorem}\label{Asplund_AP_poly}
Let $E$ be a Banach space such that $E^*$ has the RNP and AP. Then, given $N \in \N$, the set of strongly norm-attaining elements in $\Pol(^N E^*; F)$ is dense in $\Pol(^N E^*; F)$ for any Banach space $F$. 
\end{theorem} 

\begin{proof}
Let $P \in \Pol(^N E^* ; F)$ be given. Consider its corresponding bounded linear operator $T_P \in \Lin (\Nsten E^*, F)$. As $E^*$ has the RNP and the AP, we get that $\Nsten E^*$ also enjoys the RNP (see, for instance, \cite[Section 4.2]{Floret}). It follows from \cite[Theorem 5]{Bourgain} that, given $\eps >0$, there exists an absolutely strongly exposing operator $T \in \Lin (\Nsten E^*, F)$ such that $\| T - T_P \| < \eps$. Take $\mu_0 \in S_{\Nsten E^*}$ so that for every $(\mu_n) \subseteq \overline{B}_{\Nsten E^*}$ with $\| T (\mu_n)\| \rightarrow \|T \|$, there exists a subsequence of $(\mu_n)$ that converges to $\theta \mu_0$ for some $\theta \in \C$ with $| \theta | = 1$. Note that $\mu_0$ is a strongly exposed point. Thus, \cite[Proposition 1]{BR} implies that $\mu_0$ must be $\otimes^N x_0^*$ for some $x_0^* \in S_{E^*}$. 

Now, it is clear that the $N$-homogeneous polynomial $Q \in \Pol(^N E^*; F)$ which corresponds to $T$ attains its norm at $x_0^*$ and $\|Q - P\| = \|T - T_P \| < \eps$. We claim that $Q$ actually strongly attains its norm at $x_0^*$. Let $(x_n^*) \subseteq S_{E^*}$ be such that $\| Q(x_n^*) \| \rightarrow \|Q\|$. Then $\| T (\otimes^N x_n^*) \| \rightarrow \| T \|$; hence there exists a subsequence $(x_{n_j}^*)$ such that $\otimes^N x_{n_j}^* \rightarrow \lambda (\otimes^N x_0^*)$ for some $\theta' \in \C$ with $|\theta'|=1$. For fixed $u_0^* \in S_E^*$ and $u_0^{**} \in S_{E^**}$ with $u_0^{**} (u_0^*) = 1$, the map $S : \Nsten E^* \rightarrow \widehat{\otimes}_{N-1, s, \pi} E^*$ given by 
\[
S \left(\sum_{i=1}^{m} \otimes^N u_i^* \right) = \sum_{i=1}^m u_0^{**} (u_i^*) (\otimes^{N-1} u_i^*)
\]
is a well defined linear projection (see \cite[Theorem 3]{Blasco}). Thus, 
\[
S(\otimes^N x_{n_j}^*) = u_0^{**} (x_{n_j}^*) (\otimes^{N-1} x_{n_j}^*) \rightarrow \theta' S(\otimes^N x_0^*) = \theta' u_0^{**} (x_0^*) (\otimes^{N-1} x_{0}^*).
\]
Applying projections repeatedly, we can conclude, passing to a subsequence, that $(x_{n_j}^*)$ converges to $\theta'' x_0^*$ for some $\theta'' \in \C$ with $|\theta''| = 1$. So, the claim is established.  
\end{proof}

Next, we present a new denseness result of norm-attaining elements in $\A_{w^* u} (B_{E^*} ;F)$, the closed subspace of $\Ho^\infty (B_{E^*}, F)$ of weak$^*$ uniformly continuous $F$-valued functions on $B_{E^*}$. Notice that if $f \in \Ho^\infty (B_{E^*}; F)$ extends weak$^*$ continuously to $\overline{B}_{E^*}$, then $f$ becomes an element of $\A_{w^* u} (B_{E^*} ;F)$ due to the weak$^*$-compactness of $\overline{B}_{E^*}$. We will prove that if $E$ has the metric $\pi$-property, then the set of norm-attaining elements in $\A_{w^* u} (B_{E^*}; F^*)$ is dense for any Banach space $F$ which should be compared to \cite[Theorem 3.3]{AAGM}. Recall that $E$ is said to have the \emph{$\pi_\lambda$-property} if there is a net of finite rank projections $(S_\alpha)$ on $E$ converging strongly to the identity on $E$ with $\limsup_\alpha \|S_\alpha\| \leq \lambda$. When $\lambda =1$, we say that $E$ has the \emph{metric-$\pi$-property}. For a more detailed account on the $\pi_\lambda$-property, we refer to \cite{CASAZZA} or \cite{JRZ}. 

\begin{theorem}\label{thm:metric_pi}
If a Banach $E$ has the metric $\pi$-property, then the set of norm-attaining elements in $\A_{w^* u } (B_{E^*} ;F^*)$ forms a dense subset for every Banach space $F$.
\end{theorem}

Before giving the promised proof, we recall that the \emph{bounded-weak}$^*$ \emph{topology} for a dual Banach space is the largest topology that agrees with the weak$^*$-topology on bounded sets. A function of a dual Banach space is bounded-weak$^*$ continuous if its restriction to each bounded set is weak$^*$ continuous. Let us observe first the following lemma which will be used.

\begin{lemma}\label{lem:bdd_weak_star}
Let $E$ be a Banach space with the metric $\pi$-property. If $T \in \Lin (E^*, G)$ is a bounded-weak$^*$ to weak$^*$ continuous finite rank operator for some dual Banach space $G$, then, given $\eps >0$, there exists a norm one finite rank projection $S$ on $E$ such that $\|T - T \circ S^* \| < \eps$. 
\end{lemma} 

\begin{proof}
As the range of $T$ is finite dimensional, we can write $ T(x^*) = \sum_{i=1}^{N} (T^* (e_i^*))(x^*) e_i$ for some biorthogonal system $\{ e_i, e_i^*\}_{i=1}^N$ in $F \times F^*$, where $F$ is the range of $T$ which is a finite dimensional subspace of $G$. By using the bounded-weak$^*$ to weak$^*$ continuity of $T$, observe that $T^* (e_i^*)$ is indeed an element of $E$ for each $i=1,\ldots, N$. As a matter of fact, let $(x_\alpha^*)$ be a net converging weak$^*$ to $x_\infty^*$ in a bounded set in $E^*$. 
Noting $e_i^* \in F^*$ and $F^{**} = F$ for each $i = 1,\ldots, N$, we observe that  
\[
T^* (e_i^*) (x_\alpha^*) = e_i^* (T(x_\alpha^*)) \rightarrow e_i^* (T (x_\infty^*)) = T^* (e_i^*) (x_\infty^*)
\]
since $T$ is bounded-weak$^*$ to weak$^*$ continuous. Hence, $T^* (e_i^*)$ is bounded-weak$^*$ continuous functional on $E^*$. 
By \cite[Theorem V.5.6]{DS}, it follows that $T^* (e_i^*)$ is weak$^*$ continuous; thus $T^* (e_i^*)$ is an element of $E$ (see, for instance, \cite[Proposition 3.22]{FHHMZ}).

Putting $x_i = T^* (e_i^*)$ for each $i=1,\ldots, N$, we rewrite $T$ as $T(x^*) = \sum_{i=1}^{N} x^* (x_i) e_i$ for every $x^* \in E^*$. From the assumption that $E$ has the metric $\pi$-property, given $\eps >0$, there exists a norm one finite rank projection $S$ on $E$ so that $\| S (x_i) - x_i \| < \eps$ for each $i=1,\ldots, N$. For fixed $x^* \in B_{E^*}$, note that  
\begin{align*}
\left|e_i^* \left[ T(S^* (x^*)) - T(x^*) \right]\right| = |(S^* (x^*)) (x_i) - x^* (x_i)| \leq \| S (x_i) - x_i \| < \eps 
\end{align*}
for all $i = 1,\ldots, N$; hence $\sup_{1\leq i \leq N} |e_i^* (T(S^* (x^*)) - T(x^*))| < \eps$. This shows that $\| (T\circ S^*   )(x^*) - T(x^*)\| < C \eps$ for some constant $C >0$ as any two norms on a finite dimensional Banach space are equivalent. Since $x^* \in B_{E^*}$ is chosen arbitrarily, this completes the proof. 
\end{proof}

\begin{proof}[Proof of Theorem \ref{thm:metric_pi}]
Let $f$ be an element of $\A_{w^* u } (B_{E^*}; F^*)$. Then it can be extended weak$^*$ continuously to $\overline{B}_{E^*}$. From the assumption that $E$ has the metric $\pi$-property, $E$ has, in particular, the AP. By arguing as in \cite[Theorem 5.2]{ACG95}, where it is proved for scalar-valued functions, we can observe that that $f$ is the uniform limit on $\overline{B}_{E^*}$ of weak$^*$ continuous finite type polynomials from $E^*$ to $F^*$. For this reason, we shall prove the announced statement for weak$^*$ continuous finite type polynomials. 

Let $P$ be a weak$^*$ continuous finite type polynomials from $E^*$ to $F^*$. Write $P = \sum_{k=0}^N P_k$, where each $P_k$ is a $k$-homogeneous polynomial which is a linear combination of products of weak$^*$ continuous linear functionals on $E^*$. For fixed $k \in \{1, \ldots, N\}$, consider $T_k : E^* \rightarrow \Pol (^{k-1} E^*; F^*)$ given by 
\[
T_k (x^*) (z^*) = \widecheck{P}_k (x^*, z^*, \ldots, z^*),
\]
where $\widecheck{P}_k$ denotes the symmetric $k$-linear mapping from $E^* \times \cdots \times E^*$ to $F^*$ associated to $P_k$.  
It is clear that $T_k$ is a finite rank operator. For simplicity, assuming that $P_k$ is of the form $\sum_{j=1}^m x_i^k \otimes y_i^*$ for some $x_1,\ldots, x_m \in E$ and $y_1^*,\ldots, y_m^* \in F^*$, then 
\[
T_k (x^*) (z^*) = \sum_{i=1}^m x^*(x_i) z^* (x_i)^{k-1} y_i^*.
\]
Having in mind the following isometric isomorphism (see \cite{CLM})
\[
\Pol(^{k-1} E^*; F^*) = \left( \left( \widehat{\otimes}_{k-1,s,\pi} E^*\right) \widehat{\otimes}_{\pi} F \right)^*,  
\]
we have that 
\[
\left\langle (\otimes^N \varphi  )\otimes y, T_k (x^*) \right\rangle = \sum_{i=1}^m x^* (x_i) \varphi (x_i)^{k-1} y_i^* (y)
\]
for every $\varphi \in E^*$ and $y \in F$. Thus, if $(x_\alpha^*)$ is a net in $E^*$ norm bounded by $C>0$ which converges weak$^*$ to $x_\infty^*$ in $E^*$, then 
\[
\left\langle (\otimes^N \varphi  )\otimes y, T_k (x_\alpha^*) \right\rangle \rightarrow \left\langle (\otimes^N \varphi   )\otimes y, T_k (x_\infty^*) \right\rangle
\]
Recall that any element in $ \left( \widehat{\otimes}_{k-1,s,\pi} E^*\right) \widehat{\otimes}_{\pi} F$ can be approximated by a finite sum of basic tensors and observe that 
\[
\sup_{\alpha} \| T_k (x_\alpha^*)\| \leq C \| \widecheck{P}_k \| \leq C \frac{k^k}{k!} \| P \| 
\] 
where the polarization inequality and Cauchy integral formula are used (see, for instance, \cite[Proposition 1.8 and Proposition 3.2]{Dineen}). This shows that $T_k (x_\alpha^*)$ converges weak$^*$ to $T_k (x_\infty^*)$. Consequently, $T_k$ is bounded-weak$^*$ to weak$^*$ continuous. 

Let $T : E \rightarrow \oplus_{k=1}^{N-1} \Pol(^{k-1} E^* ; F^*)$ be the direct sum of $T_1, \ldots, T_N$. As we observed that each $T_k$ is a bounded-weak$^*$ to weak$^*$ continuous finite rank operator, it is clear that $T$ is also of finite rank and is bounded-weak$^*$ to weak$^*$ continuous. By Lemma \ref{lem:bdd_weak_star}, given $\eps >0$, there exists a norm one finite rank projection $S$ on $E$ such that $\| T - T \circ S^* \| < \eps$. Set each symmetric $k$-linear mapping $A_k := \widecheck{P}_k \circ (S^*, \ldots, S^*)$ and consider its corresponding $k$-homogeneous polynomial $Q_k$. 
Note that each $Q_k$ is again a linear combination of products of weak$^*$ continuous linear functionals on $E^*$.
Arguing as in the proof of \cite[Theorem 3.3]{AAGM}, we can obtain that $Q:= P_0 + \sum_{k=1}^N Q_k$ satisfies that $\| Q-P \| < N 4^N \eps$ and $Q = Q \circ S^*$. This implies that $Q$ attains its norm as $S^*$ is a norm one finite rank projection. 
\end{proof} 

It is known that the metric $\pi$-property is satisfied by a large class of Banach spaces. For instance, the following Banach spaces satisfy the property (for the complete proof, see \cite[Example 4.12]{DJRR}). 

\begin{enumerate}
\setlength\itemsep{0.3em}
\item Banach spaces with a finite-dimensional decomposition with the decomposition constant $1$;
\item $L_p (\mu)$-spaces for any $1 \leq p < \infty$ and measure $\mu$;
\item $L_1$-predual spaces; 
\item $E \otimes_a F$ whenever $E$ and $F$ satisfy the metric $\pi$-property and $|\cdot|_a$ is an absolute sum; 
\item $E \widehat{\otimes}_\alpha F$ whenever $E$ and $F$ satisfy the metric $\pi$-property and $\alpha$ is a uniform cross norm (see \cite[Section 12]{DF} for background). 
\end{enumerate}  

\begin{cor}
If $E$ is one of the above Banach spaces (1)--(5), then the set of norm-attaining elements in $\A_{w^* u } (B_{E^*} ;F^*)$ forms a dense subset for every Banach space $F$. 
\end{cor} 

\section{Holub's results in the context of holomorphic functions}\label{Holub} 

J.~R.~Holub observed in \cite{Holub} that if $E$ and $F$ are reflexive Banach spaces and one of them has the approximation property (for short, AP), then every element of $\Lin (E, F)$ attains its norm if and only if every element of $\Lin (E,F)$ is a compact operator if and only if $\Lin (E, F)$ is reflexive. For $N$-homogeneous polynomials, R.~Alencar obtained in \cite{Alencar1} a characterization for $\Pol(^N E; F)$ to be reflexive under the assumption that $E$ and $F$ both are reflexive and have the AP. Afterwards, it is generalized by J.~A.~Jaramillo and L.~A.~Moraes in \cite{JM} where only the domain space $E$ is assumed to have the AP. One year later, J.~Mujica in \cite{Mujica1} improved this result by replacing the assumption AP on $E$ by compact approximation property (for short, CAP). The main goal of this section is to present a further improvement of the result of J.~Mujica by removing the reflexivity assumption on the range space $F$ by making use of the result in \cite{DJM}, where the authors generalized Holub's aforementioned result by obtaining equivalent statements for every element of $\Lin (E,F)$ attain its norm without assuming $F$ to be reflexive.

For Banach spaces $E$ and $F$, we denote by $\Pol_f (^N E; F)$ the subspace of $\Pol(^N E; F)$ generated by finite-type polynomials from $E$ to $F$. Let $\Pol_w (^N E; F)$ denote the subspace of all $P \in \Pol (^N E; F)$ which are weak to norm continuous on bounded subsets of $E$. 



\begin{theorem}\label{thm:holub_poly}
Let $E$ be a reflexive Banach space, $F$ an arbitrary Banach space and let $N \in \N$. 
Consider the following statements. 
\begin{enumerate}[label=(\alph*)] 
\setlength\itemsep{0.4em}
\item $\Pol (^N E; F) = \overline{ \Pol_f (^N E; F)}$.
\item $\Pol(^N E; F) = \Pol_w (^N E; F)$.
\item $\Pol(^N E; F) = \NA \Pol (^N E; F)$.
\item $ (\Pol(^N E; F), \| \cdot \|)^* = (\Pol(^N E; F), \tau_c)^*$.
\item $\Pol(^N E; F)$ is reflexive. 
\end{enumerate} 
Then the following implications hold. 
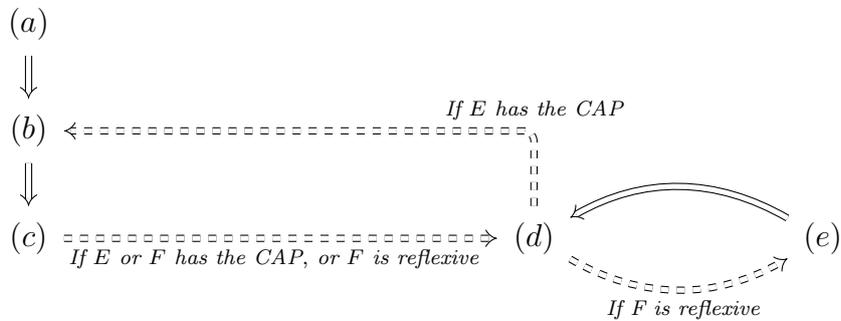
\begin{center}
	\begin{tikzpicture}[scale=0.95, baseline=0cm]
	\tikzstyle{caixa} = [rectangle, rounded corners, minimum width=.7cm, minimum height=.7cm, text centered, draw=black]
	\tikzstyle{dfletxa} = [double equal sign distance,-implies, shorten >= 2pt, , shorten <= 2pt]
		\tikzstyle{dfletxalr} = [double equal sign distance, implies-implies, shorten >= 2pt, , shorten <= 2pt]
	\tikzstyle{dfletxadotted} = [double equal sign distance, dashed, -implies, shorten >= 2pt, , shorten <= 2pt]
	
	\node (A) at (-7,0) {$(a)$};
	\node (B) at (-7,-1.5) {$(b)$};
	\node (C) at (-7,-3) {$(c)$};
	\node (D) at (0,-3)  {$(d)$};
	\node (E) at (4,-3) {$(e)$};
	
	\draw [dfletxa] (A) -- (B);
	\draw [dfletxa] (B) -- (C);
	%
	
	\draw[dfletxadotted, ->, >=implies, rounded corners] (C.east)
	--  (-2,-3) node[below, pos=.65] {$ \scriptstyle \text{If } E \text{ or } F \text{ has the CAP}, \text{ or } F \text{ is reflexive}$}
	-- (D.west);
	
	\draw[dfletxadotted, ->, >=implies, rounded corners] (D.north)
	-- (0,-2)
	-- node[above, pos=1] {$ \scriptstyle \text{If } E \text{ has the CAP}$} (0,-1.5)
	-- (B.east);
	
	\draw[dfletxadotted, ->, >=implies, rounded corners] (D)  to   [bend right]   node[below] {$\scriptstyle \text{ \text{If} F \text{is reflexive}}$} (E);
	\draw[dfletxa, ->, >=implies, rounded corners] (E)  to   [bend right]   node[below] {} (D);

	\end{tikzpicture}	
\end{center}

\end{theorem} 

\begin{proof}
(a) $\Rightarrow$ (b): It follows from the well-known results that $\Pol_f (^N E; F) \subset \Pol_w (^N E; F)$ and $\Pol_w (^N E; F)$ is closed in $\Pol(^N E; F)$ (see, for instance, \cite[Proposition 2.4]{AP}). 

(b) $\Rightarrow$ (c): Let $P \in \Pol (^N E; F)$ and take a sequence $(x_n)$ in $S_E$ so that $\| P (x_n) \| \rightarrow \|P \|$. By using the Eberlein-\v{S}mulyan theorem \cite[Theorem 3.109]{FHHMZ}, passing to a subsequence, if necessary, we may assume that $(x_n)$ converges weakly to some $x_0$ in $\overline{B}_E$. As $P$ is weakly continuous, we conclude that $P$ attains its norm at $x_0$; thus (c) follows. 

(c) $\Rightarrow$ (d):
From the assumption, we have that every element of $\Lin (\Nsten E, F)$ attains its norm; hence $\Nsten E$ is reflexive. Suppose first that $E$ or $F$ has the CAP. 
If $E$ has the CAP, then by \cite[Corollary 7]{Caliskan}, so does $\Nsten E$. Thus, \cite[Theorem B]{DJM} shows that 
$(\Lin  (\Nsten E, F) , \| \cdot \|)^* = (\Lin  (\Nsten E, F) , \tau_c)^*$. Now, \cite[Theorem 3.2]{Mujica1} proves that (d) holds. Next, assume that $F$ is reflexive. Then from the fact that every element of $\Lin (\Nsten E, F)$ attains its norm, we can deduce that $\Lin (\Nsten E, F) = ((\Nsten E) \pten F^*)^*$ is reflexive by James' reflexivity theorem. Now, \cite[Lemma 2.3]{Mujica1} shows that $(\Lin  (\Nsten E, F) , \| \cdot \|)^* = (\Lin  (\Nsten E, F) , \tau_c)^*$; so the conclusion is fulfilled. 

(d) $\Rightarrow$ (b): This follows from the proof of $(4) \Rightarrow (1)$ of \cite[Theorem 3.1]{Mujica1}.     

(d) $\Rightarrow$ (e): 
As we have seen, (d) implies $(\Lin  (\Nsten E, F) , \| \cdot \|)^* = (\Lin  (\Nsten E, F) , \tau_c)^*$. Notice that there is a canonical isomorphism 
\[
(\Nsten E) \pten F^* = (\Lin (\Nsten E, F), \tau_c)^*
\]
(see, for instance, \cite[Section 5.5, page 62]{DF}), which actually yields a topological isomorphisms:
\begin{equation}\label{eq:topological_isomorphism}
(\Nsten E) \pten F^* = (\Lin (\Nsten E, F), \tau_c)_b^*,
\end{equation}
where the $(\Lin (\Nsten E, F), \tau_c)_b^*$ is the strong dual space of $(\Lin (\Nsten E, F), \tau_c)$. Since $F$ is reflexive, $\Lin (\Nsten E ,F)$ is isometrically isomorphic to $((\Nsten E) \pten F^*)^*$. This along with \eqref{eq:topological_isomorphism} shows that $\Lin (\Nsten E, F)$ is reflexive (for the detailed proof, see the proof of $(3) \Leftrightarrow (4)$ of \cite[Theorem 2.1]{Mujica1}, where the reflexivity assumption on a domain space was not used). 

(e) $\Rightarrow$ (d): As $\Lin (\Nsten E, F)$ is reflexive, \cite[Lemma 2.3]{Mujica1} and \cite[Theorem 3.2]{Mujica1} show that (d) holds. 
\end{proof} 

\begin{remark}\label{rem:wsc_continuity}
Let us denote by $\Pol_{wsc} (^N E; F)$ the space of all $N$-homogeneous polynomials which map weakly convergent sequences in $E$ to convergent sequences in $F$. Let us also denote by $\Pol_{wu} (^N E; F)$ the subspace of all $P \in \Pol(^N E; F)$ which are weakly uniformly continuous on bounded sets. 
Let us mention that it is known \cite[Theorem 2.9]{AHV} that $\Pol_{wu} (^N E; F)$ coincides with $\Pol_w (^N E; F)$. Notice from the proof of (b) $\Rightarrow$ (c) that we can see that what is actually needed is the weakly sequentially continuity of polynomials in $\Pol (^N E; F)$. 
As a matter of fact, it is shown in \cite[Proposition 2.12]{AHV} that if a Banach space $E$ does not contain a copy of $\ell_1$, then $\Pol_{wu}(^N E; F) = \Pol_{wsc} (^N E; F)$ for any Banach space $F$; hence $\Pol_w (^N E; F) = \Pol_{wsc} (^N E; F)$. Consequently, if a reflexive Banach space $E$ has the CAP, then $\Pol (^N E; F) = \Pol_w (^N E; F)$ if and only if $\Pol (^N E; F) = \Pol_{wsc} (^N E;  F)$, which is equivalent to that $\Pol (^N E; F)$ is reflexive whenever $F$ is reflexive. On the other hand, it is worth to mentioning that the proof of (b) $\Rightarrow$ (c) can be obtained from a more general observation, that is, any \emph{weak to norm continuous function} $f$ from a reflexive Banach space $E$ to a Banach space $F$ attains its norm. Indeed, the map $x \mapsto \| f(x)\|$ is a weakly continuous function from $E$ to $\mathbb{R}$. Using the weak compactness of $\overline{B}_E$, we deduce that there exists $x_0 \in \overline{B}_E$ so that $\|f \| = \| f(x_0)\|$. 
\end{remark}

Given Banach spaces $E$ and $F$, we denote by $\Hb (E; F)$ the space of entire functions $f : E \rightarrow F$ which are bounded on bounded subsets of $E$. Let $\Ho_{wu} (E; F)$ be the subspace of $\Hb (E; F)$ of all entire functions $f$ which are weakly uniformly continuous on each bounded subset of $E$. In \cite{JM}, J.~A.~Jaramillo and L.~A.~Moraes extended Holub's result to the context of vector-valued holomorphic mappings by showing that $\Hb (E;F)$ is reflexive if and only if $\Hb (E; F) = \Ho_{wu} (E;F)$ if and only if for every $f \in \Hb (E; F)$ there is $x \in S_E$ so that $\|f(x)\| = \sup \{ \| f (u) \| : u \in \overline{B}_E \}$ provided that both $E$ and $F$ are reflexive and $E$ has the AP. Let us mention that, earlier than the appearance of \cite{JM}, R.~Ryan in \cite{Ryan_thesis} and Y.~S.~Choi and S.~G.~Kim in \cite{CK95} obtained equivalent statements concerning the reflexivity of spaces of scalar-valued holomorphic functions. Using Theorem \ref{thm:holub_poly}, we improve the result of J.~A.~Jaramillo and L.~A.~Moraes by weakening the AP  assumption on $E$ and removing the reflexivity assumption on $F$.

\begin{theorem}\label{thm:holub_holo}
Let $E$ be a reflexive Banach space and $F$ be an arbitrary Banach space. Consider the following statements. 
\begin{enumerate}[label=(\alph*)] 
\setlength\itemsep{0.4em}
\item $\Pol(^N E; F) = \Pol_w(^N E; F)$ for every $N \in \N$. 
\item $\Hb (E; F) = \Ho_{wu} (E; F)$.
\item For every $f \in \Hb (E; F)$ and $r >0$, there exists $x_0 \in r S_E$ such that $\| f (x_0) \| = \sup \{ \| f (u) \| : u \in r \overline{B}_E \}$. 
\item $\Pol(^N E; F) = \NA \Pol (^N E; F)$ for every $N \in \N$. 
\item $\Hb (E;F)$ is reflexive.  
\end{enumerate} 
Then the following implications hold.
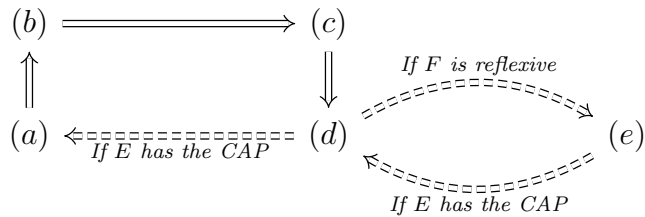
\begin{center}
\begin{tikzcd}
(b) \arrow[rrr, Rightarrow] &  &  & (c) \arrow[d, Rightarrow]                                                                                              &  &  &     \\
(a) \arrow[u, Rightarrow]   &  &  & (d) \arrow[dashed, lll, "\text{If } E \text{ has the CAP}", Rightarrow] \arrow[dashed, rrr, "\text{If } F \text{ is reflexive}", Rightarrow, bend left] \arrow[dashed, rrr, "\text{If } E  \text{ has the CAP}"', Leftarrow, bend right] &  &  & (e)
\end{tikzcd}
\vspace{3mm}
\end{center} 
\end{theorem}  
 
\begin{proof}
(a) $\Rightarrow$ (b): It is known that if $f \in \Hb (E; F)$ has the Taylor expansion $\sum_{N=0}^\infty P_N$, then $f \in \Ho_{wu} (E; F)$ if and only if $P_N \in \Pol_{wu} (^N E; F)$ for every $N \in \N$ (see \cite[Proposition 1.5]{Aron}). Recall from \cite[Theorem 2.9]{AHV} that $\Pol_w (^N E; F) = \Pol_{wu} (^N E; F)$; hence (a) implies (b).  

(b) $\Rightarrow$ (c): Let $f \in \Hb (E; F)$ and $r >0$ be given. Arguing as in the proof of (b) $\Rightarrow$ (c) in Theorem \ref{thm:holub_poly}, we obtain an element $x_0 \in r S_E$ such that $\| f(x_0) \| = \sup \{ \| f (u) \| : u \in r \overline{B}_E \}$. 

(c) $\Rightarrow$ (d): It is clear. 

(d) $\Rightarrow$ (a): It follows from Theorem \ref{thm:holub_poly}. 

(d) $\Rightarrow$ (e): If $F$ is reflexive, then as in the proof of (c) $\Rightarrow$ (d) in Theorem \ref{thm:holub_poly}, we have that $\Pol (^N E; F)$ is reflexive for every $N \in \N$. Then \cite[Corollary 3]{AnsemilPonte} proves that $\Hb (E; F)$ is reflexive. 

(e) $\Rightarrow$ (d): Suppose that $E$ has the CAP. Since $\Pol (^N E; F)$ is a closed subspace of $\Hb (E; F)$, we have that $\Pol (^N E; F)$ is reflexive for every $N \in \N$. In particular, this implies that $F$ must be reflexive. Now, (d) follows from Theorem \ref{thm:holub_poly}. 
\end{proof} 


\section{Improvements of some results on polynomial reflexivity}\label{Farmer} 

In this section, we shall observe that some of results on polynomial reflexivity due to J.~Farmer which rely on the AP of a domain space can be generalized to the case when the domain space has the CAP thanks to Theorem \ref{thm:holub_poly}. 

Given $N \in \N$, recall that a Banach space $E$ is said to be \emph{$\mathcal{P}_N$-reflexive} if the space $\Pol(^N E)$ is reflexive. If this is true for every $N\in \N$, then $E$ is said to be \emph{polynomially reflexive}. 
It is observed by R.~Alencar, R.~M.~Aron and S.~Dineen  \cite[Theorem 6]{AAD} that the Tsirelson space $T^*$ (for its definition, see \cite{Tsi}) is polynomially reflexive. Afterwards, this result was generalized by Farmer in \cite{Farmer94} by using the theory of spreading models. As we will not use the theory of spreading models, we do not give its definition but refer to \cite{B, BS1}.   

As promised, we present the following result which generalizes \cite[Theorem 1.3]{Farmer94}. Since the proof can be obtained by following the original proof step by step but having in mind from Remark \ref{rem:wsc_continuity} that for a reflexive space $E$ with CAP, the reflexivity of $\Pol(^N E)$ is equivalent to that $\Pol(^N E) = \Pol_{wsc} (^N E)$, we omit the proof.


\begin{proposition}
Suppose that $E$ is a reflexive space. Then if no spreading model built on a weakly null sequence has a lower $q$-estimate for any $q < \infty$, then any subspace of $E$ with the CAP is polynomially reflexive. 
\end{proposition} 

Some results concerning the polynomially Schur property in \cite{Farmer94} also can be improved. 
Recall from \cite{FJ93} that a Banach space is \emph{polynomially Schur} (resp., \emph{$\mathcal{P}_N$-Schur}) if whenever a sequence converges to zero against every scalar-valued homogeneous polynomial (resp., every scalar-valued $N$-homogeneous polynomial) then it must converge to zero in norm. 

\begin{proposition}
\begin{enumerate} 
\setlength\itemsep{0.4em}
\item Suppose that $E$ is polynomially reflexive and has the CAP, and that $F$ is polynomially Schur. Then $\Pol (^N E; F) = \Pol_w (^N E; F)$ for every $N \in \N$. 
\item Given $N,M \in \N$, suppose that $E$ is $\Pol_{NM}$-reflexive and has the CAP, and that $F$ is $\Pol_N$-Schur, then  $\Pol (^M E; F) = \Pol_w (^M E; F)$. 
\end{enumerate}  
\end{proposition}

\begin{proof}
(1): Let $P \in \Pol (^N E; F)$ and let $(x_n)$ be a sequence in $E$ which converges weakly to zero.  As $E$ is polynomially reflexive, by Remark \ref{rem:wsc_continuity}, we can deduce that $P(x_n)$ converges to zero against every scalar-valued homogeneous polynomial on $F$. Since $F$ is polynomially Schur, we have then $P(x_n)$ converges to zero in norm; hence $P \in \Pol_{wsc} (^N E; F)$; hence $P \in \Pol_{w} (^N E; F)$ by \cite[Proposition 2.12]{AHV} and \cite[Theorem 2.9]{AHV}. 

(2): Observe that the $N$-fold projective tensor product of a $\Pol_{NM}$-reflexive space is $\Pol_M$-reflexive. Using this fact and arguing as the above (1), the conclusion follows.
\end{proof}


Let us denote by $\mathcal{A} (B_E)$ the subalgebra of $\A_u (B_E)$ which is generated by the constant functions and elements of $E^*$. Also, let us consider the open unit ball $B_E$ as a subset of $\mathcal{A}(B_E)^*$ via evaluations. From \cite[Theorem 4.1]{Farmer94}, some equivalent statements of the polynomial reflexivity are obtained by comparing the topology $\sigma (\mathcal{A} (B_E)^*, \mathcal{A} (B_E)^{**})$ with the weak topology on $rB_E$ for $r <1$ under the AP assumption on $E$. We would like to improve this result replacing the AP condition by the CAP. At this time, we give a detailed proof since we slightly modify the argument.

Observe first the following simple lemma which will be used in the proof of Theorem \ref{thm:gen_Farmer_1}. 


\begin{lemma}\label{lem:unif_conti_poly}
Let $E$ and $F$ be Banach spaces and $N \in \N$. If $T \in \mathcal{K} (E, E)$ and $P \in \Pol(^N E; F)$, then $P \circ T$ belongs to $\Pol_{wu} (^N E; F)$. 
\end{lemma} 

\begin{proof}
To complete the proof, it is enough to show that $P \circ T$ is weakly uniformly continuous on the open unit ball $B_E$ of $E$. Let $\eps >0$ be given. Noting that $P$ is uniformly continuous on bounded sets of $E$, we can choose $\delta > 0$ such that if $u, v \in \| T\| \overline{B}_E$ satisfies $\|u - v \| < \delta$ then $\| P(u) - P(v) \| < \eps$. On the other hand, by \cite[Proposition 2.5]{AP}, $T$ is weakly uniformly continuous on bounded subsets of $E$. Thus, we can find $x_1^*, \ldots, x_m^* \in E^*$ and $\eta >0$ such that if $x, y \in B_E$ with $|x_i^* ( x - y) | < \eta$ for all $i=1,\ldots, m$, then $\| T(x) - T(y) \| < \delta$; hence $\|P(T(x)) - P(T(y))\| < \eps$. This shows that $P \circ T$ is weakly uniformly continuous on $B_E$. 
\end{proof}

\begin{theorem}\label{thm:gen_Farmer_1}
Let $E$ be a reflexive Banach space. Consider the following statements.
\begin{enumerate}[label=(\alph*)] 
\setlength\itemsep{0.4em}
\item $E^*$ has an unconditional spreading model built on a weakly null sequence with an upper $p$-estimate for some $p > 1$. 
\item The $\sigma (\mathcal{A} (B_E)^*, \mathcal{A} (B_E)^{**})$-topology does not agree with the weak topology on $r B_E$ for $r < 1$. 
\item The $\sigma (\mathcal{A}_{wu} (B_E)^*, \mathcal{A}_{wu} (B_E)^{**})$-topology does not agree with the weak topology on $r B_E$ for $r < 1$.  
\item There is an $N$-homogenous polynomial on $E$ which is not weakly sequentially continuous. 
\end{enumerate} 
Then we always have that (a) $\Rightarrow$ (b) $\Rightarrow$ (c) $\Rightarrow$ (d). If, in addition, $E$ has the CAP, then (d) $\Rightarrow$  (c) and (d) is equivalent to that $E$ is not $\mathcal{P}_N$-reflexive. And if $E$ has the AP, then (d) $\Rightarrow$ (b). 
\end{theorem} 

\begin{proof}
(a) $\Rightarrow$ (b), and (d) $\Rightarrow$ (b) when $E$ has the AP are proved in \cite[Theorem 4.1]{Farmer94}. 

(b) $\Rightarrow$ (c): Assume to the contrary that $\sigma (\mathcal{A}_{wu} (B_E)^*, \mathcal{A}_{wu} (B_E)^{**})$-topology agrees with the weak topology on $r B_E$ for $r < 1$. Let $\iota : \mathcal{A} (B_E) \hookrightarrow \mathcal{A}_{wu} (B_E)$ be the natural inclusion map. Suppose that $(x_\alpha)$ is a net in $r B_E$ so that $(x_\alpha)$ converges weakly to some $x_\infty$ in $r B_E$. Let $\varphi \in \mathcal{A} (B_E)^{**}$ be fixed, then $\iota^{**} (\varphi)$ belongs to $\mathcal{A}_{wu} (B_E)^{**}$. 
It follows that $\varphi (x_\alpha) = \iota^{**} (\varphi) (x_\alpha) \rightarrow \iota^{**} (\varphi) (x_\infty) = \varphi (x_\infty)$. This implies that the $\sigma (\mathcal{A} (B_E)^*, \mathcal{A} (B_E)^{**})$-topology also agrees with the weak topology on $r B_E$.

(c) $\Rightarrow$ (d): Assume to the contrary that every homogeneous polynomial on $E$ is weakly sequentially continuous. We claim that if a net $(x_\alpha)$ in $r B_E$ converges weakly to some $x_\infty$ in $r B_E$, then $(x_\alpha)$ also tends to $x_\infty$ in the $\sigma (\mathcal{A}_{wu} (B_E)^*, \mathcal{A}_{wu} (B_E)^{**})$-topology. Let $\psi \in \mathcal{A}_{wu} (B_E)^{**}$ be fixed and consider its restriction $\psi \vert_{B_E}$ to $B_E \,\,(\subset \mathcal{A}_{wu} (B_E)^*)$. Then $\psi \vert_{B_E}$ can be seen as an element of $\Ho^\infty (B_E)$. Thus, we can write its Talyor series on $B_E$: 
\[
\psi \vert_{B_E} = \sum_{N=1}^\infty P_N \quad \text{with} \quad \| P_N \| \leq \|\psi \vert_{B_E} \|. 
\]
It is well-known that its partial sums $(\sum_{N=1}^M P_N )_{M=1}^\infty$ converges uniformly to $\psi \vert_{B_E}$ on $r B_E$. Given $\eps >0$, choose $M \in \N$ large enough so that $\| \sum_{N={M+1}}^\infty P_N \|_{r B_E} < \eps/4$. 
As $E$ is reflexive, it does not contain a copy of $\ell_1$; hence our assumption implies that $P_N$ is weakly (uniformly) continuous on bounded sets \cite[Proposition 2.12]{AHV}. It follows that there exists $\alpha_0$ so that $\alpha \geq \alpha_0$ implies that $|P_N (x_{\alpha}) - P_N (x_\infty)| < \eps/(2M)$ for $N=1,\ldots, M$. Now, it follows that 
\[
| \psi \vert_{B_E} (x_\alpha) - \psi \vert_{B_E} (x_\infty) | \leq \left| \sum_{N=1}^{M} P_N (x_\alpha) - \sum_{N=1}^{M} P_N (x_\infty) \right| + \frac{\eps}{2} < \eps
\]
for all $\alpha \geq \alpha_0$. 

(d) $\Rightarrow$ (c): Suppose that $E$ has the CAP. Let $P \in \Pol(^N E)$ which is not weakly sequentially continuous at the origin. Choose a sequence $(x_i)$ converging weakly to zero with $|P(x_i)| \geq \eps > 0$ for all $i \in \N$. Pick any sequence $(\eps_i)$ of positive real numbers tending to zero. By \cite[Proposition 1 and Remark 1]{CJ85}, there exists $T_n \in \mathcal{K} (E, E)$ with $\|T_n \| \leq 1$ be such that $\| T_n (x_i) - x_i \| < \eps_n$ for every $i =1,\ldots, n$, for each $n \in \N$. By Lemma \ref{lem:unif_conti_poly}, we have that $P \circ T_n$ belongs to $\A_{wu} (B_E)$. Let $F$ be a weak$^*$ limit of the set $\{ P \circ T_n \}_{n \in \N}$ in $\A_{wu} (B_E)^{**}$. Let us consider the modulus $\eps \mapsto \delta(\eps)$ of the uniform continuity of $P$ on $B_E$, i.e., if $x, y \in B_E$ satisfies $\| x - y \| < \delta(\eps)$, then $| P(x)-P(y) | <\eps$. Given $\delta' > 0$ and $k \in \N$, let us choose $i > k$ so that 
\[
| F(x_k) - P T_i (x_k) | \leq \delta(\eps_i) + \delta'. 
\]
Letting $i \rightarrow\infty$, we have that $|F(x_k) - P(x_k) | \leq \delta'$ for every $k \in \N$. This implies that $F(x_k) \not\rightarrow 0$ as $k \rightarrow \infty$.  
\end{proof}

\proof[Acknowledgements]

The author would like to thank Yun Sung Choi, Sheldon Dantas, Han Ju Lee, Manuel Maestre and Abraham Rueda Zoca for fruitful conversations on the topic of the paper, and wishes to thank Miguel Mart\'in for pointing out an error in the first version of the paper. The author is also indebted to the anonymous referees for their helpful comments and suggestions toward the improvement of the paper.

\end{document}